\documentclass[a4paper,12pt]{article}
\usepackage{amssymb,fancyhdr,amsmath,theorem}

\usepackage{multicol}

\usepackage[T1]{fontenc}

\usepackage{color}

\usepackage[final]{graphicx}

\newtheorem{lemma}{Lemma}

\newtheorem{theorem}{Theorem}

\newtheorem{claim}{Claim}

\newtheorem{conjecture}{Conjecture}

\newtheorem{corollary}{Corollary}

\newenvironment{proof}
   {\begin{trivlist}\item[]\textbf{\bf{Proof. }}\ignorespaces}
   {\qed\end{trivlist}}

\newcommand{\qed}{{\ifmmode q.e.d. \else\unskip\nobreak\hfil
\penalty50\quad\null\nobreak\hfill $\square$ \parfillskip=0pt
\finalhyphendemerits=0\par\fi}} 

\title{Vertex coloring of plane graphs with nonrepetitive boundary paths}
\author{J\'anos Bar\'at\thanks{Research is supported by OTKA Grants PD~75837 and K~76099, and the J\'anos Bolyai Research Scholarship of the
Hungarian Academy of Sciences.}\\
\small Department of Computer Science and Systems Technology\\[-0.8ex]
\small University of Pannonia, Egyetem u.\ 10, 8200 Veszpr\'em, Hungary\\
\small \texttt{barat@dcs.vein.hu}\\
\and
J\'ulius Czap\\
\small Department of Applied Mathematics and Business Informatics\\ [-0.8ex]
\small Faculty of Economics, Technical University of Ko\v{s}ice\\ [-0.8ex]
\small B. N\v{e}mcovej 32, SK-040 01 Ko\v{s}ice, Slovakia\\
\small \texttt{julius.czap@tuke.sk}
}
\date{}
\begin{document}

\maketitle

\begin{abstract}
A sequence $s_1,s_2,\dots,s_k,s_1,s_2,\dots,s_k$ is a repetition. 
A sequence $S$ is nonrepetitive, if no subsequence of consecutive terms of $S$ form a repetition. 
Let $G$ be a vertex colored graph. 
A path of $G$ is  nonrepetitive, if the sequence of colors on its vertices is nonrepetitive. 
If $G$ is a plane graph, then a facial nonrepetitive vertex coloring of $G$ is a vertex coloring such that any facial path is nonrepetitive. 
Let $\pi_f(G)$ denote the minimum number of colors of a facial nonrepetitive vertex coloring of $G$. 
Jendro\v l and Harant posed a conjecture that $\pi_f(G)$ can be bounded from above by a constant.
We prove that $\pi_f(G)\le 24$ for any plane graph $G$. 
\end{abstract}

MSC: 05C15
\section{Introduction}

A sequence $s_1,s_2,\dots,s_k,s_1,s_2,\dots,s_k$ is a {\it repetition}. 
A sequence $S$ is {\it nonrepetitive}, if no subsequence of consecutive terms of $S$ form a repetition. 
A vertex coloring of $G$ is {\it nonrepetitive}, if there is no path $v_1,v_2,\dots,v_{2t}$ such that $v_i$ and $v_{t+i}$ receive the same color 
for all $i=1,2,\dots,t$. 
The {\it Thue chromatic number} of a graph $G$ is the minimum number of colors needed in a nonrepetitive coloring.
It is denoted by $\pi(G)$.
The seminal result in this field is by Thue \cite{T}, who proved that the $n$-vertex path $P_n$ satisfies

\begin{theorem}
$\pi(P_1)=1$, $\pi(P_2)=\pi(P_3)=2$, and $\pi(P_n)=3$ for every $n\ge4$.
\end{theorem}

A major open problem in this area is due to Grytczuk \cite{G}.

\begin{conjecture}\label{conjG}
There is an absolute constant $K$ such that any planar graph $G$ satisfies $\pi(G)\le K$.
\end{conjecture}

Reportedly, several group of authors have achieved wrong-proven results with the constant $K$ between 2000 and 3000. 
Bar\'at and Varj\'u \cite{BV} showed that the general constant is at least 10. 
Motivated by Conjecture \ref{conjG}, Bre\v sar, Grytczuk, Klav\v zar, Niwczyk and Peterin \cite{BGKNP} proved

\begin{theorem}
If $T$ is a tree, then $\pi(T)\le 4$, and the bound is tight.
\end{theorem}

The next step towards planar graphs includes the outerplanar graphs. 
Independently, Bar\'at, Varj\'u \cite{BV} and K\"undgen, Pelsmajer \cite{KP} proved a reasonably good bound.

\begin{theorem}\label{BV}
If $G$ is an outerplanar graph, then $\pi(G)\leq 12$.
\end{theorem}

Considering a conjecture for planar graphs, one can usually get a good insight by looking at 
the $k\times k$ grid.
In this case, the $k\times k$ grid is known to have bounded Thue chromatic number \cite{BV}.
The {\em tree-width} of a graph $G$ can be defined to be the minimum integer $k$ such that $G$
is a subgraph of a chordal graph with no clique on $k+2$ vertices.
One prominent feature of the grid is that a large grid has large tree-width.
Therefore, it was a natural complementary idea to consider the class of graphs with bounded tree-width.
Independently, Bar\'at, Varj\'u \cite{BV} and K\"undgen, Pelsmajer \cite{KP} proved an upper bound
exponential in the tree-width, but independent of the number of vertices.

\begin{theorem}
If $G$ is a graph of tree-width $t$, then $\pi(G)\le 4^t$.
\end{theorem}

A {\it facial path} consists of consecutive vertices on the boundary of a face.
Considering the nonrepetitive property, we may restrict our attention for facial paths only.
Jendro\v l and Harant \cite{JH} introduced precisely this notion. 
If $G$ is a connected plane graph, a {\it facial nonrepetitive vertex coloring} of $G$ is a vertex coloring such that any facial path 
is nonrepetitive. 
Here a {\it plane graph} is a graph together with a fixed embedding in the plane.
The {\it facial Thue chromatic number} of $G$, denoted by $\pi_f(G)$, is the minimum number of colors of all facial nonrepetitive vertex colorings of $G$. 
Jendro\v l and Harant \cite{JH} proved the following

\begin{theorem}
If $G$ is a $2$-connected plane graph of maximum degree $\Delta$, $\Delta \geq 3$, then
\begin{itemize}
\item $\pi_f(G) \leq 120 \ln\Delta$,
\item $\pi_f(G) \leq \min\{29(\Delta-2),39\sqrt{\Delta-2},47\sqrt[3]{\Delta-2}\}$,
\item $\pi_f(G) \leq 29$ if $\Delta=3$, $\pi_f(G) \leq 44$ if $\Delta=4$, $\pi_f(G) \leq 59$ if $\Delta=5$, $\pi_f(G) \leq 67$ if $\Delta=6$, $\pi_f(G) \leq 73$ if $\Delta=7$,
\item $\pi_f(G) \leq 16$ if $G$ is Hamiltonian.
\end{itemize}
\end{theorem}

Based on their experience, Jendro\v l and Harant posed the following

\begin{conjecture}\label{conjJH}
If $G$ is a plane graph, then $\pi_f(G)\le K$, for some constant~$K$.
\end{conjecture}

In the next section, we prove this conjecture, with $K=24$.
Our idea is to use Theorem~\ref{BV} recursively for layers of the targeted plane graph.
The subtlety is to introduce some extra edges, when complications may occur.

Since our ideas work with local colorings, we can naturally extend the result to
graphs embedded in surfaces.

We also prove that the $n\times n$ grid requires at most 4 colors for a 
facial nonrepetitive coloring.

\section{Results}

Let $H$ be a plane graph. 
Let $\partial (H)$ denote the set of vertices, which are incident with the outer face of $H$. 
Let $[\partial (H)]$ denote the subgraph of $H$ induced by $\partial (H)$.

\bigskip
The main result of this paper is the following

\begin{theorem} \label{24}
If $G$ is a plane graph, then $\pi_f(G)\leq 24$. 
\end{theorem} 

\begin{proof}
First, we color the vertices of $G$ with blue and red as follows:
\begin{itemize}
\item Set $G_1=G$, and color the vertices of $\partial (G_1)$ with blue. 
\item Set $G_{i+1}=G_i \setminus \partial (G_i)$, and color the vertices of $\partial (G_{i+1})$ with red, 
if the vertices of $\partial (G_i)$ were blue, and color with blue, if the vertices of $\partial (G_i)$ were red.
\end{itemize} 

\begin{claim}
The graph $[\partial (G_i)]$ is outerplanar for every $i$. 
\end{claim} 

Let $B$ denote the set of blue vertices of $G$, and $R$ the red ones. 
A path $v_1 \dots v_j v_{j+1} \dots v_k v_{k+1} \dots v_m$ is $BRB$, if $v_1, \dots, v_j \in B$,  $v_{j+1}, \dots ,v_k \in R$, and $v_{k+1}, \dots ,v_m \in B$.
We similarly define the $RBR$ paths.

Let $B(f)$ and $R(f)$ denote the set of blue and red vertices incident with face $f$.

\begin{claim}\label{bfrf}
If a face $f$ of $G$ is incident with both a blue and a red vertex, 
then there is a subscript $i$ such that either $B(f)\subseteq \partial (G_i)$ and $R(f)\subseteq \partial (G_{i+1})$ 
or $R(f)\subseteq \partial (G_i)$ and $B(f)\subseteq \partial (G_{i+1})$.
\end{claim}

Let $f$ be a face of $G$.
\begin{itemize}
\item if $B(f)\subseteq \partial (G_i)$ and $R(f)\subseteq \partial (G_{i+1})$, 
then for any $BRB$ path\\ $v_1 \dots v_j v_{j+1} \dots v_k v_{k+1} \dots v_m$, 
where $v_1, \dots ,v_j \in B$,  $v_{j+1}, \dots ,v_k \in R$,  $v_{k+1}, \dots ,v_m \in B$, we insert an edge $v_jv_{k+1}$ inside the face $f$.
\item if $R(f)\subseteq \partial (G_i)$ and $B(f)\subseteq \partial (G_{i+1})$, 
then for any $RBR$ path\\ $v_1 \dots v_j v_{j+1} \dots v_k v_{k+1} \dots v_m$, 
where $v_1, \dots ,v_j \in R$,  $v_{j+1}, \dots ,v_k \in B$,  $v_{k+1}, \dots ,v_m \in R$, we insert an edge $v_jv_{k+1}$ inside the face $f$.
\end{itemize}

In this way, we obtain a new graph $[\partial (G_i)]^+$ from the graph $[\partial (G_i)]$. 
Observe, that $[\partial (G_i)]^+$ is also outerplanar.
Therefore, it has a nonrepetitive vertex coloring with at most $12$ colors, see Theorem \ref{BV}.

Let us color the vertices of $[\partial (G_i)]^+$ nonrepetitively with colors $1,\dots, 12$ for $i$ odd, 
and color the vertices of $[\partial (G_i)]^+$ nonrepetitively with colors $13,\dots, 24$ for $i$ even.
Now every (facial) path in $[\partial (G_i)]^+$ is nonrepetitive. 

These colorings induce a coloring of $G$ with $24$ colors. 
It remains to show that there is no repetitive boundary path in $G$.

Assume that there is a face $f$ with a repetitive boundary path. 
We may assume $B(f)\subseteq \partial (G_i)$ and $R(f)\subseteq \partial (G_{i+1})$, see Claim \ref{bfrf}.

The repetitive boundary path can not be monochromatic, since otherwise the same path is repetitive in $[\partial (G_i)]^+$ or in $[\partial (G_{i+1})]^+$. 
Therefore, we may assume that the path is $BR \dots BR$. 
In this case, the complete $B \dots B$ part is a repetitive path in $[\partial (G_i)]^+$, a contradiction.
\end{proof} 

As a lower bound for the class of planar graphs, we observe the following 

\begin{lemma}
 There exists a planar graph $G$ with $\pi_f(G)=5$.
\end{lemma}

\begin{proof}
 For instance, the $5$-wheel is such a graph. Indeed, a $5$-cycle needs $4$ colors. 
If we introduce a new vertex $v$ in the middle adjacent to all five vertices, then it needs a fifth color.
\end{proof}

In the proof of Theorem~\ref{24}, we use colorings locally. 
This suggests, that our proof can be extended to graphs on surfaces.

\begin{theorem}
 If $G$ is a graph embedded on the torus, then $\pi_f(G)\leq 24^2$. 
\end{theorem}

\begin{proof}
If $C=x_1,\dots x_k,x_1$ is a noncontractible cycle of $G$, then let us cut the torus
along $C$. In this way, we get a plane graph $G^*$ with vertex set $V(G)\setminus V(C)$ plus two copies of $V(C)$. 
The faces of $G^*$ are the same as of $G$ with two additional faces formed by the two copies of $V(C)$.
We use Theorem~\ref{24} for $G^*$ and get a coloring $c^*$ using 24 colors.
We need a coloring $c$ of $G$ based on $c^*$.
The only trouble is to color $V(C)$, since $c^*$ possibly associated different colors to the two copies. 
Otherwise $c$ and $c^*$ are the same.
For a vertex $x_i$, if the two different colors associated by $c^*$ were $s$ and $t$, then $c(v):=(s,t)$.
There are $24^2-24$ such ordered pairs. 
Since $c^*$ was facial nonrepetitive, so does $c$.
The total number of colors we used is $24^2$.
\end{proof}

Using the same idea, we can prove the following

\begin{corollary}
 If $G$ is a graph, which is embedded in a surface of genus $g$, then
$\pi_f(G)\le 24^{2^g}$.
\end{corollary}

We recall a definition from \cite{MT}.
Let $C_1,\dots,C_k$ be a collection of pairwise disjoint cycles in an embedded graph $G$.
These cycles form a {\it planarizing collection of cycles} if the cutting along all of 
$C_1,\dots,C_k$ results in a connected graph embedded in the sphere (plane).
If there exists a planarizing collection of cycles, then we get the same bound as for the torus.

\begin{corollary}
 If $G$ is a graph, which is embedded in a surface of genus $g$, and there exists a 
planarizing collection of cycles, then $\pi_f(G)\leq 24^2$. 
\end{corollary}

A string, that is a finite sequence of symbols, is {\it palindromic}, if it reads the same backwards.
In the proof of our next theorem, we use a palindromic nonrepetitive sequence.
We need some definitions and recall a Theorem by Thue.
Upper bar denotes the complement of a string modulo $3$, that is 
$\overline {a_1a_2\dots a_n}=(2-a_1)(2-a_2)\dots(2-a_n)$,
where $a_i \in\{0,1,2\}$.

Let $\beta_i$ be a string defined recursively as follows: $\beta_0=2$ and 

$\gathered
\beta_i=\left\{
\aligned
\beta_{i-1} 2 \overline {\beta_{i-1}}&\quad\text{for an even $i$},\\
\beta_{i-1} 1 \overline {\beta_{i-1}}&\quad\text{for an odd $i$}.
\endaligned
\right.
\endgathered$

Thue \cite{T} proved the following

\begin{theorem}\label{beta}
As defined above, $\beta_{2k}$ is a palindromic nonrepetitive sequence for any $k \geq 0$.
\end{theorem}

Observe that $\beta_{2k}=\beta_{2k-2} 1 \overline {\beta_{2k-2}} 2 \overline {\beta_{2k-2}} 1 \beta_{2k-2}$ for any $k \geq 1$.

There exists a facial nonrepetitive coloring of a planar triangulation with at most $4$ colors, since 
only the sequences of length two need to be checked for the nonrepetitive property.
Naturally, longer faces might create more trouble. 
For a typical planar graph with a lot of $4$-cycles, we prove that still $4$ colors are sufficient.

\begin{lemma}
 For the $n\times n$ grid $G_{n\times n}$, $\pi_f(G_{n\times n})\le 4$.
\end{lemma}

\begin{proof}
Let us assume that $n$ is even, say $n=2n_1$. 
The vertices of $G_{n\times n}$ can be covered by $n_1$ concentric cycles $C_1, \dots ,C_{n_1}$, where
$C_1$ is the outer cycle of the embedding.
We refer to the position of the vertices according to the next figure.

\begin{center}
 \includegraphics[scale=0.4]{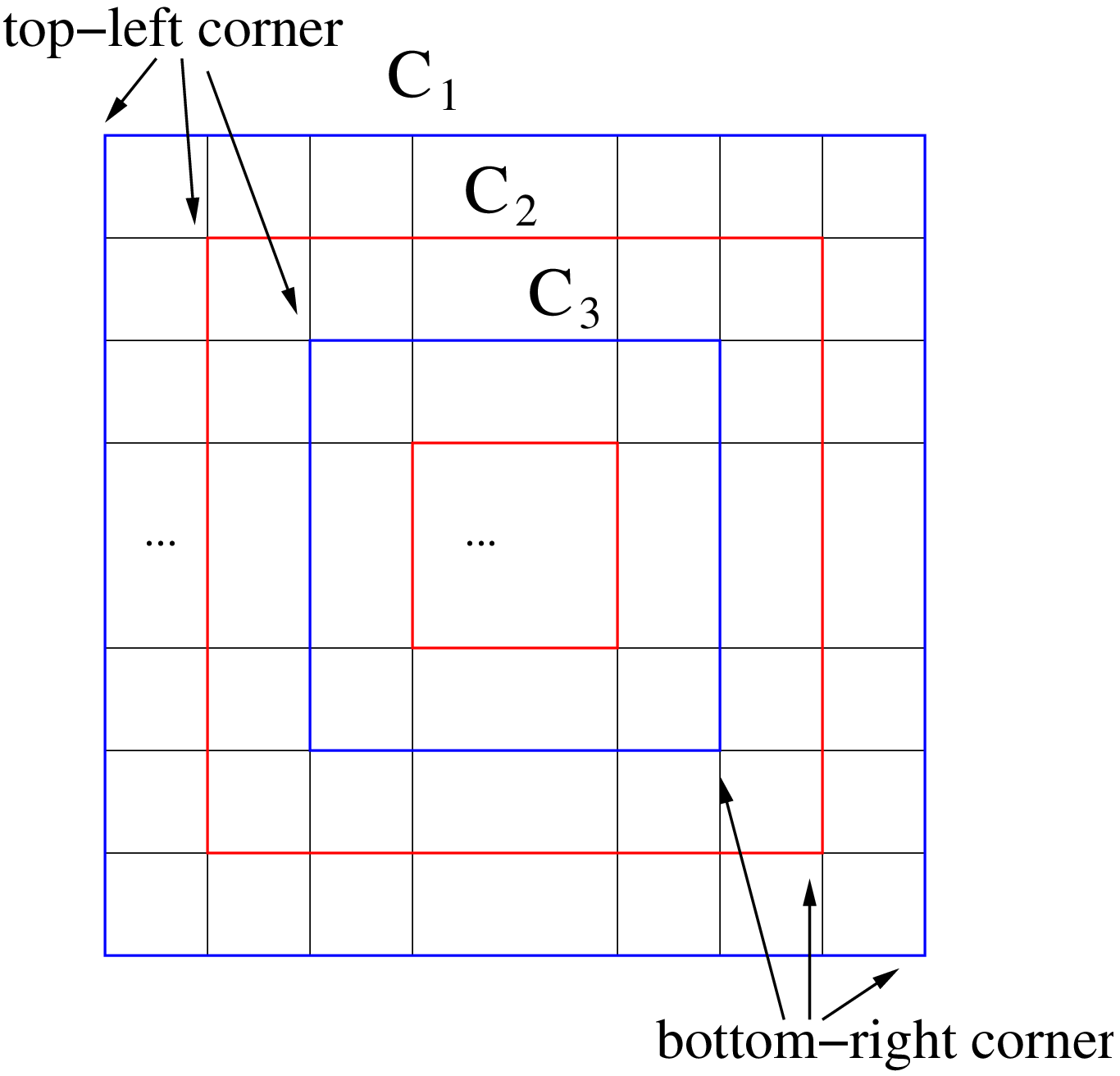}
\end{center}

We color the vertices according to these cycles, going inward.
First, we color the vertices of the outer cycle starting at the top-left corner and going clockwise.
We use the same coloring going counterclockwise.
The coloring is determined by a palindromic nonrepetitive sequence $s=\beta_{2k}$ of length at least $4n-5$  from Theorem~\ref{beta}, where the middle is the top-left corner.
We fix the sequence $s$ for the whole process, and use smaller and smaller finite parts of it
for the cycles $C_1,C_3,\dots$.
We use a forth color, $3$ say, for the bottom-right corner.

In the second round, we consider the next concentric cycle, $C_2$.
We color every second vertex by $3$ starting
with the top-left corner being blank.
The color of the top-left corner is determined by our fixed sequence, in our case it receives color $1$.
The bottom-right corner of $C_2$ depends on the colors of the facial neighbors in $C_1$.
However, there are two choices, we pick one.
We jump to the next level to color $C_3$.
We want to repeat the idea of coloring clockwise and counterclockwise with the palindromic nonrepetitive sequence.
We use the sequence for the bottom-right corner also, instead of using color $3$.
Now the colors of the corners of $C_2$ are forbidden for the corners of $C_3$.
Still there are two choices for each corner.
Therefore, we can select the colors to be different or to be the same, as we wish.
We do it according to the sequence we fixed in the beginning, using a part of it of appropriate length.
We permute the colors $\{0,1,2\}$, if we need.

In the next step, we color the blank vertices of $C_2$.
The main idea behind the construction and coloring the vertices in concentric
cycles comes now. 
Any blank vertex needs to get a color different from its neighbors, and this is the only constraint.
There is no danger of constructing a repetitive $4$-path.
Since every blank vertex is adjacent to two vertices with color $3$, there are at most two 
forbidden colors. Therefore, one of $\{0,1,2\}$ is available for each blank vertex.

From now on, we can repeat the above steps for each pair of consecutive cycles.
The only thing we have to check is the last or last two cycles.
If $n_1$ is odd, then we finish without trouble.
If $n_1$ is even, then we have to show that the last four vertices in the middle can be colored.
For the penultimate cycle, we use the sequence $2,0,1,2,0,2,1$.
Here again, take into account the freedom of choice for the corner vertices and the permutation of the colors.
Now the last four vertices can be colored according to the next figure, which finishes the proof
of the even case.

\begin{center}
 \includegraphics[scale=0.5]{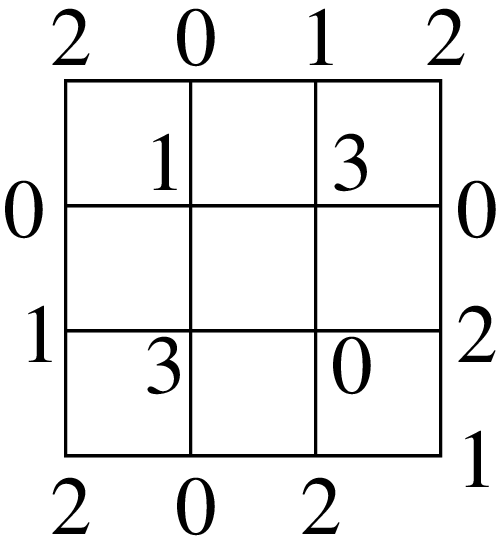}
\end{center}

The odd case is very similar. Suppose that $n=2n_1+1$.
We consider the concentric cycle cover $C_1,\dots C_{n_1},C_{n_1+1}$ again.
Here $C_{n_1+1}$ is a single vertex.
We color $C_1$ with the fixed nonrepetitive palindromic sequence.
The only exception is the bottom-right corner, which gets color $3$.
Now we partly color $C_2$, starting at the top-left corner, which is colored by $3$.
{}From that on, we color every second vertex by $3$ and leave the rest blank.
Again, the bottom-right corner is left blank, exceptionally.
Now we color the $2\times 2$ bottom-right corner of $C_2$ and $C_3$ according to the next figure.

\begin{center}
 \includegraphics[scale=0.5]{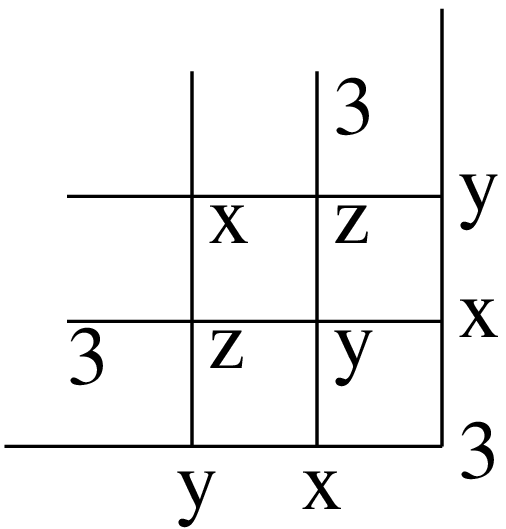}
\end{center}

We now jump to $C_3$. There is one precolored vertex.
Anyway, we can use the fixed nonrepetitive palindromic sequence, possibly with permuted colors.
Now we turn back to $C_2$ and finish the coloring as before.

These steps can be repeated, and we color the concentric cycles in pairs.
At the end, the remaining single vertex is either colored by $3$ or $2$ depending on parity. 
Both cases work without any trouble.
\end{proof}


\begin{thebibliography}{99}

\bibitem{BV} J. Bar\'at, P.P. Varj\'u, On square-free vertex colorings of graphs, Studia Sci Math Hungar 44 (2007), no. 3, 411--422.

\bibitem{BGKNP} B. Bre\v sar, J. Grytczuk, S. Klav\v zar, S. Niwczyk, I. Peterin, Non-repetitive colourings of trees, Discrete Math 307 (2007), 163--172.


\bibitem{G} J. Grytczuk, Pattern avoidance on graphs, Discrete Math 307 (2007), 1341--1346.

\bibitem{JH} S. Jendro\v l, J. Harant, Nonrepetitive vertex colorings of graphs, preprint 2011, 12 pages.

\bibitem{KP} A. K\"undgen, M.J. Pelsmajer, Non-repetitive colourings of graphs of bounded tree-width, Discrete Math 308 (2008), 4473--4478.

\bibitem{MT} B. Mohar, C. Thomassen, Graphs on Surfaces, The John Hopkins University Press, Baltimore and London, 2007.

\bibitem{T} A. Thue, \"Uber unendliche Zeichenreichen, Norske Vid Selsk Skr, I Mat Nat Kl, Christiana 7 (1906), 1--22 (In German).
\end{thebibliography}
\end{document}